\documentclass[journal,twoside,web]{ieeetran}

\usepackage{todonotes}

\usepackage[utf8]{inputenc} 
\usepackage[T1]{fontenc}    
\usepackage{url}            
\usepackage{booktabs}       
\usepackage{amsfonts}       
\usepackage{nicefrac}       
\usepackage{microtype}      
\usepackage{amsthm}
\usepackage{cite}

\usepackage{amsmath,amssymb}
\usepackage{algorithmic,algorithm}
\usepackage{nccmath}
\usepackage{mathtools}
\usepackage{xcolor}
\usepackage{graphicx} 

\usepackage{wrapfig}
\usepackage{enumitem}
\usepackage{subfig}
\usepackage{cleveref}

\newcommand{\tr}{\mathrm{tr}}

\renewcommand{\baselinestretch}{0.98}

\makeatletter

\newcommand{\E}{\mathbb E}

\newcommand{\R}{\mathbb{R}}

\newtheorem{lemma}{Lemma}[section]

\newtheorem{theorem*}{Theorem}[section]
\newtheorem{corollary}{Corollary}[section]

\title{
\textbf{\texttt{InterQ:}} A DQN Framework for Optimal Intermittent Control 
}

\author{Shubham Aggarwal$^*$, Dipankar Maity$^*$, Tamer Ba{\c s}ar \vspace{-1cm}
\thanks{* denotes equal contribution. Research of the authors was supported by the ARL grant ARL DCIST CRA W911NF-17-2-0181. Research of SA and TB was also supported in part by the ARO grant W911NF-24-1-0085.
}
\thanks{
Shubham Aggarwal and Tamer Ba{\c s}ar are with the Coordinated Science Laboratory at the University of Illinois Urbana-Champaign, Urbana, IL, USA-61801. 
Dipankar Maity is with Department of Electrical and Computer Engineering at the University of North Carolina at Charlotte, Charlotte, NC, USA-28223.
(Emails:\texttt{\{sa57, basar1\}@illinois.edu, dmaity@charlotte.edu)}
}
}

\begin{document}

\maketitle
\thispagestyle{empty}
\begin{abstract}
In this \textit{letter}, we explore the communication-control co-design of discrete-time stochastic linear systems through reinforcement learning. Specifically, we examine a closed-loop system involving two sequential decision-makers: a scheduler and a controller. The scheduler continuously monitors the system's state but transmits it to the controller intermittently to balance the communication cost and control performance.
The controller, in turn, determines the control input based on the intermittently received information. Given the partially nested information structure, we show that the optimal control policy follows a certainty-equivalence form. Subsequently, we analyze the qualitative behavior of the scheduling policy. To develop the optimal scheduling policy, we propose $\texttt{InterQ}$, a deep reinforcement learning algorithm which uses a deep neural network to approximate the Q--function. Through extensive numerical evaluations, we analyze the scheduling landscape and further compare our approach against two baseline strategies: (a) a multi-period periodic scheduling policy, and (b) an event-triggered policy. The results demonstrate that our proposed method outperforms both baselines. The open source implementation can be found at \url{https://github.com/AC-sh/InterQ}.
\end{abstract}


\section{Introduction}

\IEEEPARstart{I}{ntermittent} control systems have been extensively studied over the past two decades, finding applications in diverse fields such as power markets, biological control, communication networks, and multi-agent systems \cite{wen2015event,chakrabarty2017event,li2021real}. These systems typically involve two decision-making entities: a \textit{sampler} and a \textit{controller}, operating within a team setting. The controller determines control inputs based on measurements received from the sampler, whose objective is to minimize the communication burden on the system.

Research in this area has primarily focused on two objectives: \textit{event-triggered} control \cite{heemels2021event} and joint \textit{optimal control-communication} co-design \cite{molin2014optimal}. In event-triggered control, the primary goal is to stabilize the dynamical system, with triggering conditions derived using Lyapunov-based analysis. In contrast, the second approach aims to optimize a given performance objective, jointly determining both the optimal controller and the optimal scheduling policy. In this work, we adopt the second approach of optimal design.

A major long-standing challenge in this framework is characterizing the optimal scheduling policy. Theoretical research in this domain faces difficulties because incorporating both control and sensing/communication decisions into the strategy set leads to a decentralized information structure, where the scheduler's actions influence the available information to the controller. 
Furthermore, the control problem dictated by the scheduling decisions is \textit{bilinear} in the system state (which corresponds to the estimation error) and the binary control variable (which corresponds to the scheduling decision), and involves a quadratic minimization objective. 
State-of-the-art results \cite{soleymani2021value,soleymani2022value} establish that the optimal scheduling policy belongs to the class of threshold policies, but its exact form depends on the value function, which remains unknown. Consequently, computing scheduling instants within an optimal control framework still remains an open problem.

In this work, we tackle this problem using deep reinforcement learning (RL). Specifically, we approximate the state-action value function at the scheduler using a neural network and employ the deep Q-Learning algorithm \cite{mnih2015human} to determine the optimal scheduling instants. The neural network is trained using the same quadratic cost function mentioned earlier. Our results indicate that the proposed policy framework outperforms both periodic triggering policies and event-triggered policies from the literature.

\textbf{Related Work:} Research on optimal decision-making in teams dates back to the foundational works of Marschak and Radner \cite{marschak1955elements,radner1962team}, which studied settings involving multiple decision-makers (DMs), each with access to different information variables. In such scenarios, the DMs independently choose policies while collectively aiming to minimize a common cost or maximize a shared reward. Since the information available to each DM varies, the design of joint optimal policies is highly dependent on the underlying information structure \cite{yuksel2013stochastic}, particulary when the information evolves dynamically, as demonstrated in seminal works by Witsenhausen \cite{witsenhausen1968counterexample}, Feldbaum \cite{feldbaum1961dual}, and Ba\c{s}ar \cite{basar2008variations}. 

The optimal intermittent control problem falls within this class of decentralized decision-making problems, where the scheduler and controller together form a two-decision-maker (2-DM) team. There has been extensive research on deriving joint controller-scheduler policies in such settings \cite{imer2005optimal,imer2010optimal,lipsa2011remote,molin2009lqg,maity2020minimal}. Notably, works such as \cite{imer2010optimal,lipsa2011remote} establish the optimality of threshold policies for scalar systems. Possible conjectures for the multivariate case were later discussed in \cite{molin2014optimal}. More recently, studies in \cite{soleymani2021value,soleymani2022value} extend this analysis to multi-dimensional Gauss-Markov processes, proving that the optimal estimator remains linear and invariant to no-communication events—a property previously established for scalar systems in \cite{imer2010optimal}, with its optimality formally proven in \cite{lipsa2011remote}. This finding enables the design of a globally optimal control policy; however, the scheduler policy remains dependent on the value function, which is unknown and cannot be computed exactly.

Motivated by these challenges in scheduler policy characterization for multivariate systems, a major contribution of this work is to address the open problem of computing optimal scheduling policies for multivariate systems using deep RL techniques in intermittent control problems involving remotely located controller-scheduler pairs. Specifically, we propose \textbf{\texttt{InterQ}}--a deep Q-learning algorithm, where the state-action value function is approximated using a multi-layer perceptron, which is trained using the bilinear estimation error dynamics and a quadratic loss function. Our extensive numerical analysis and ablation studies demonstrate that our approach outperforms both periodic scheduling policies and event-triggered policies.

\textbf{Organization: }The rest of the \textit{letter} is organized as follows. We formulate the problem in Section \ref{sec:Problem_form}. The optimal controller and scheduler policies are computed in Sections \ref{sec:OptController} and \ref{sec:optSensor}, respectively.
We provide supporting numerical simulations in Section \ref{sec:numSims} and conclude the paper with major highlights in Section \ref{sec:conc_disc}.

\textbf{Notations:} We denote the trace of a square matrix $X$ by $\tr(X)$. For two symmetric matrices $X,Y$, the notation $X \succ (\succeq) Y$ implies that $X-Y$ is positive (semi-)definite.

\section{Communication-Constrained Control}\label{sec:Problem_form}
Consider a controlled dynamical system which constitutes a triple $(\mathcal{P}, \mathcal{S}, \mathcal{C})$, where $\mathcal{P}$ represents the dynamical system, and $\mathcal{S}$ and $\mathcal{C}$ denote the scheduler and the controller, respectively. The state $x_k \in \mathbb R^{n_x}$ of $\mathcal{P}$ evolves according to the following discrete-time stochastic linear difference equation:
\begin{align}\label{eq:plant}
x_{k+1} = Ax_k + B u_k + W_k, ~~k \ge 0,
\end{align}
where $u_k \in \mathbb R^{n_u}$ denotes the control input and $W_k \in \mathbb R^{n_x}$ denotes the additive system noise, all at time $k$. We assume that the noise has zero mean and finite covariance $K_W$; the assumption on zero mean is not necessary, but only simplifies the presentation. Further, the initial state $x_0$ is also sampled from a zero-mean finite-covariance distribution. We take the noise and the initial state to be independent of each other and across time. 
In contrast to the majority of the existing work, we do not assume these noises to be Gaussian.

The state $x_k$ is controlled over a feedback loop by two active decision-makers (the scheduler and the controller) which may not be geographically colocated. This means that the scheduler, which continuously monitors the state $x$, can schedule to transmit this state value over a network to the controller for applying a control input. The running cost incurred by the controller is quadratic and is given~as:
\begin{align}\label{eq:cntrl_obj}
c(x,u):=\|x\|_Q^2 + \|u\|^2_R
\end{align}
with $Q \succeq 0$ and $R \succ 0$ denoting appropriate dimensional matrices. 

Next, in contrast to perfect-state optimal control problems \cite{anderson2007optimal}, our formulation focuses on costly communication channel between the scheduler and the controller. As a result, the scheduler must actively decide on when to convey information to the controller. Classic examples of such a setup include load frequency control in power systems \cite{wen2015event}, scheduled control in artificial pancreas \cite{chakrabarty2017event}, and intermittent control of inverted pendulum modeled human motor control \cite{gawthrop2011intermittent}, among others.

Let $\lambda > 0$ denote the scheduling cost per time instant whenever a communication happens between the scheduler and the controller.
Moreover, let us define the set of (possibly random) scheduling instants up to time $k$ as
\begin{align}
    {\tt T}_k := \{k_{\ell} \mid \ell = 1, \ldots, n_k\},
\end{align}
where $n_k$ is the total number of communications up to time $k$. 
Henceforth, we will let ${\tt T} := \lim_{k\to \infty} {\tt T}_k$. Let $a_k \in \{0,1\}$ denote the scheduling decision variable at time $k$.
That is, $k \in {\tt T}$ if, and only if, $a_{k} = 1$.

As a result, the incurred communication cost at time $k$ is given as:
\begin{align}\label{eq:sensing_cost}
    m_k := \lambda a_k.
\end{align}
Combining both the control and communication objectives from \eqref{eq:cntrl_obj} and \eqref{eq:sensing_cost}, the overall objective function is given as:
\begin{align}\label{eq:overall_cost}
    & J  = \mathbb E \Big[\! \sum_{k=0}^\infty \gamma^k (c(x_k,u_k) + \alpha m_k) \Big].
\end{align}
where $\gamma \in (0, 1]$ is the discount factor and $\alpha > 0$ is the tradeoff parameter balancing the control and scheduling costs. 
We may absorb the $\alpha$ within $m$ by defining new weights $\bar\lambda = \alpha \lambda$. 
Consequently, without loss of any generality, we assume $\alpha = 1$ for the remainder of this \textit{letter}. 
Finally, the expectation in \eqref{eq:overall_cost} is taken with respect to the stochasticity induced by the initial state, the system noise, and the possible randomization in the control and scheduling policies. 

Let us define the information available to the controller and the scheduler at time $k$ as
\begin{align*}
    {\tt I_{C}}(k) & := \{x_s, u_r \mid s \in {\tt T}_k, r \in [k\!-\!1]\},
    ~{\tt I_{C}}(0) \!:=\! \emptyset,  \\
    {\tt I_{S}}(k) & := \{x_s,u_r,{\tt T}_k \mid s \in [k], r \in [k-1]\},
    ~{\tt I_{S}}(0) \!:=\! \emptyset,
\end{align*}
where we have used the notation $[j]:= \{0, 1, \cdots, j\}$.

Let $\mu_c: {\tt I_{C}} \rightarrow \mathbb R^{n_u}$ denote a control policy which lies in the space of admissible control policies defined as $\mathcal{M}_{C}: = \{\mu_{C} \mid \mu_{C} $ is adapted to the sigma-field generated by ${\tt I_{C}}\}$. Further, let $\mu_s: {\tt I_{S}} \rightarrow \mathbb \{0,1\}$ denote a scheduler policy which lies in the space of admissible scheduling policies defined as $\mathcal{M}_{S}: = \{\mu_{s} \mid \mu_{s} $ is adapted to the sigma-field generated by ${\tt I_{S}}\}$. 
With the introduction of the above policies, we are now interested in deriving the pair of optimal controller and scheduler policies $(\mu_{c}^*, \mu_{s}^*)$ which satisfy 
\begin{align}\label{eq:overall_problem}
    J^* &:= J(\mu_{c}^*, \mu_{s}^*) = \min_{(\mu_{c}, \mu_{s}) \in \mathcal{M}_C \times \mathcal{M}_S} J(\mu_{c}, \mu_{s}).
\end{align}

\section{Optimal Control Policy} \label{sec:OptController}
We begin this section by noting that the information sets of the controller and scheduler exhibit a partially nested structure:
\begin{align}\label{eq:nested_info}
    {\tt I_C}(k) \subseteq {\tt I_S}(k), ~~ \forall k \ge 0.
\end{align}
As a consequence, the optimization objective in \eqref{eq:overall_problem} can be rewritten as:
\begin{align}\label{separation_eqn}
J^* = \min_{\mu_s \in \mathcal{M}_S} \min_{\mu_c \in \mathcal{M}_C} J(\mu_c, \mu_s).
\end{align}
This formulation follows from the no-dual effect property of control \cite{feldbaum1961dual}. In dynamic sequential decision-making problems, such as the one considered in this work, a controller typically serves two purposes:

1) Feedback control: Directly applying a control action, and 

2) Reducing future uncertainty: choosing inputs that enhance future state estimation.

The above is often referred to as the dual effect of control \cite{feldbaum1961dual}.
However, due to \eqref{separation_eqn}, the scheduler fully determines when the state is transmitted to the controller. Consequently, the controller's role is restricted to exerting control actions, eliminating the dual effect. This leads to the emergence of a \textit{separation principle}, as indicated in \eqref{separation_eqn}. As a result, we can first compute the optimal controller and subsequently the optimal scheduler.

In the sequel, we will assume that the pair $(A,B)$ is controllable and the pair $(A, Q^{\nicefrac{1}{2}})$ is observable. Then, to derive the optimal controller, we first rewrite the objective in \eqref{eq:overall_cost} using completion of squares:
\begin{align}\label{eq:compl_squares}
    J & = \!\frac{\gamma}{1-\gamma}\tr( P K_W) + \E \Big[ \!\sum_{k=0}^\infty  \gamma^k \|u_k \!+\! \gamma \hat R^{-1}B^\top P A x_k\|_{\hat R}^2 \Big] \nonumber \\
    & \qquad +  \mathbb E \left [\sum_{k=0}^\infty \gamma^k m_k \right],
\end{align}
where $\hat R := R+\gamma B^{\top }PB$ and $P \succeq 0$ is the unique positive semi-definite solution to the following algebraic Riccati equation
\begin{align}\label{eq:ARE}
P = \gamma A^{\top }PA - \gamma^2 A^{\top }PB\hat R^{-1}B^{\top }PA + Q.
\end{align}
Thus, with sampled measurements, the optimal controller takes the form
\begin{align} \label{eq:optimalControl}
    u^*_k : = \mu^*_c({\tt I_C}(k)) = - \gamma \hat R^{-1}B^\top P A \hat{x}_k, ~\forall k \in [T]
\end{align}
where  $\hat{x}_k = \E[x_k \mid {\tt I_C}(k)] $ is the least-squares estimate for $x_k$ under the information available to the controller. Furthermore, using the zero-mean property of the system noise, we let the estimate $\hat x_k$ follow the dynamics:
\begin{align*}
    & \hat{x}_{k+1}  = (1-a_{k+1}) (A \hat{x}_k + B u_k) + a_{k+1} x_{k+1}, ~\forall k \geq 0
\end{align*}
with $\hat x_0 = (1-a_0)\E[x_0] + a_0 x_0$. Under the optimal control strategy of \eqref{eq:optimalControl}, we may further simplify the estimator dynamics to obtain 
\begin{align} \label{eq:OptimalEstimator}
    & \hat{x}_{k+1}  \!= \!(1-a_{k+1})(A \!-\! \gamma B \hat R^{-1} B^\top P A )\hat{x}_k + a_{k+1} x_{k+1}, ~\forall k.
\end{align} 
This completes our discussion of the optimal control policy. Next, we turn toward constructing an optimal scheduler policy using the deep RL framework.

\section{Optimal Scheduler Policy}\label{sec:optSensor}
Let us begin by defining the estimation error at the controller as $e_k = x_k - \hat x_k$. Consequently, the error dynamics can be written using \eqref{eq:plant} and \eqref{eq:OptimalEstimator} as
\begin{align}\label{eq:est_error}
    e_{k+1} = (1-a_{k+1}) (Ae_k + W_k), ~~\forall k 
\end{align}
with $e_0 = (1-a_0)(x_0 - \E [x_0])$. We note that the evolution of error $e$ is fully determined by the scheduling policy $\mu_s$. Next, we substitute the optimal control law \eqref{eq:optimalControl} into \eqref{eq:compl_squares} to obtain the final expression for the closed-loop cost function:
\begin{align} \label{eq:OptJ}
J= \frac{\gamma}{1-\gamma} \tr(P K_W) + \E \left[\sum_{k=0}^\infty \gamma^k (\|e_k\|_{\Gamma}^2 + \lambda a_k )\right],
\end{align}
where we defined $\Gamma:= \gamma^2 A^\top P B \hat R^{-1} B^\top PA$.
 Accordingly, we can introduce a Markov Decision Process (MDP) ${\tt M}:= ({\tt S}, {\tt A}, {\tt P}, {\tt C})$, where ${\tt S}$ denotes the state space, ${\tt A}$ denotes the action space, ${\tt P}$ denotes the state transition dynamics and ${\tt C}$ denotes the running cost incurred when an action is executed. The state of MDP is the estimation error, with its transition dynamics ${\tt P}$ given by the difference equation \eqref{eq:est_error}, and thus ${\tt S} = \mathbb R^{n_x}$. The set of actions is that of scheduling actions $\{0,1\}$. Finally, the running cost is given by ${\tt C}(e,a) = \|e\|^2_\Gamma + \lambda a$. Subsequently, for a given policy $\mu_s$, we define the state-action value function $Q^{\mu_s}(e,a): {\tt S} \times {\tt A} \rightarrow \mathbb R$ of the above MDP as
\begin{align}
    Q^{\mu_s}(e,a) = \E \left[\sum_{k = 0}^\infty \gamma^k {\tt C}(e_k,a_k) \mid \mu_s, e_0 = e, a_0 = a \right].
\end{align}

 Using the Bellman equation, for any time $k$, we write the optimal state-action value function as
 \begin{align} \label{eq:Q-func}
     Q(e_k,a_k) & = \inf_{\mu_s \in \mathcal{M}_S} Q^{\mu_s}(e_k,a_k) \nonumber \\
     & = \E[ {\tt C}(e_k, a_k) + \gamma \min_{a' \in {\tt A}} Q(e_{k+1}, a')].
 \end{align}
 Subsequently, one may compute the optimal scheduling policy~as
 \begin{align}\label{eq:Q_policy}
     \mu_s^*({\tt I_S}(k)) = \arg\min_{a_k \in {\tt A}} Q(e_k,a_k). 
 \end{align}
In the given definition, the closed-form expression of the state-action value function remains unknown, making it impossible to compute the scheduling policy exactly. Thus, in this work, we employ RL to train an ``agent'' to interact with the environment and learn an optimal policy that minimizes cumulative costs. An associated challenge with exact policy computation even with RL is the continuous nature of the state space which prevents the exact computation of this value function, unlike tabular MDPs with a finite state space \cite{sutton1998reinforcement, wei2020model, vamvoudakis2021handbook}. To overcome this issue, function approximation techniques are necessary to (approximately) represent the Q-function within a chosen function class, which can then be used to derive an approximate policy using \eqref{eq:Q_policy}. Commonly used function classes in the literature \cite{sutton1998reinforcement} include linear functions, polynomial functions, and radial basis functions, among others. In this work, however, we utilize a deep neural network (DNN) to approximate the Q-function, leading to the deep Q-learning method \cite{mnih2013playing,mnih2015human} (which we detail in the next to the following subsection).
However, before proceeding with the deep-Q learning framework, we will derive some sub-optimal yet effective policies which will later be used for verifying the efficiency of the learned policy. 

\subsection{Suboptimal Scheduling Policies} 

For this discussion, let us define and simplify the Bellman value function:
\begin{align} \label{eq:Ve}
    V(e)  & = \min_{a\in {\tt A}} Q(e,a) \nonumber\\
    & = \min \{Q(e,0), Q(e,1)\} \nonumber \\
    & \overset{\eqref{eq:Q-func}}{=} \min\{\|e\|^2_\Gamma + \gamma \E[ V(Ae + w)],~~\|e\|^2_{\Gamma}+\lambda + \gamma V(0)\} \nonumber\\
    & = \|e\|^2_\Gamma + \min\{ \gamma \E[ V(Ae + w)],~~\lambda + \gamma V(0)\}, 
\end{align}
where the expectation in $\E[V(Ae+w)]$ is taken over the randomness of $w$. 
Now, picking $e=0$, one may obtain an upper bound on $V(0)$ to be 
\begin{align} \label{eq:V0}
    V(0) \le \nicefrac{\lambda}{(1-\gamma)}.
\end{align}
Subsequently, we present the following lemma which characterizes the qualitative behavior of the scheduling landscape.
\begin{lemma}
    A sufficient condition for not scheduling a communication at an error state $e$ is 
    \begin{align}\label{eq:noschedule}
        \|e\|^2_{A^\top \Gamma A} < \lambda\left( \nicefrac{1}{\gamma} -1 \right) - \tr(\Gamma K_W),
    \end{align}
    and a sufficient condition for scheduling a communication is 
    \begin{align} \label{eq:toschedule}
        \|e\|^2_{A^\top \Gamma A} > \nicefrac{\lambda}{\gamma(1-\gamma)} - \tr(\Gamma K_W).
    \end{align}
\end{lemma}

\begin{proof}
    Notice that we may compactly write the $Q$-function~as 
    \begin{align*}
        Q(e,a) = \|e\|^2_\Gamma + a(\lambda + \gamma V(0)) + (1-a)\gamma \E [V(Ae+w)].
    \end{align*}
    Consequently, $a^*=0$ is \textit{strictly} optimal at $e$ if, and only if 
    \begin{align} \label{eq:inequality1}
        \E [V(Ae+w)] < \nicefrac{\lambda}{\gamma} + V(0).
    \end{align}
    Using \eqref{eq:Ve} on the LHS of \eqref{eq:inequality1} yields 
    \begin{align} \label{eq:inequality2}
        \E[\|Ae+w\|^2_\Gamma] + & \min\{ \gamma \E[ V(A^2e + Aw + w')],~\lambda + \gamma V(0) \} \nonumber \\
        &< \nicefrac{\lambda}{\gamma} + V(0),
    \end{align}
where the expectation in $ \E[\|Ae+w\|^2_\Gamma]$ is taken over the randomness of $w$ and that in $\E[ V(A^2e + Aw + w')]$ is taken over the randomness of $w$ and $w'$.
Notice that 
\begin{align} \label{eq:inequality3}
       \|e\|^2_{A^\top \Gamma A} + \tr(\Gamma K_W) +\lambda + \gamma V(0) < \nicefrac{\lambda}{\gamma} + V(0)
\end{align}
is a sufficient condition to satisfy \eqref{eq:inequality2}, where the latter is sufficient to ensure $a^*=0$ at $e$, thus making \eqref{eq:inequality3} a sufficient condition for no scheduling at $e$. 
Since $V(0) \ge 0$, one may verify that \eqref{eq:noschedule} is sufficient for not scheduling a transmission. 

To obtain the sufficient conditions for making $a^*=1$ optimal, we change the inequality directions in \eqref{eq:inequality1} and \eqref{eq:inequality2}. 
Furthermore, by noticing that $V(e) \ge 0$ for all $e\in \R^{n_x}$, one may verify that 
\begin{align} \label{eq:inequality4}
       \|e\|^2_{A^\top \Gamma A} + \tr(\Gamma K_W) > \nicefrac{\lambda}{\gamma} + V(0)
\end{align}
is a desired sufficient condition. 
Further, using \eqref{eq:V0}, we conclude that \eqref{eq:toschedule} is a sufficient condition for $a^*=1$ at $e$. This completes the proof.
\end{proof}

\begin{corollary}
    The optimal scheduling strategy is to always schedule if the scheduling cost satisfies  $\lambda \le \gamma(1-\gamma)\tr(\Gamma K_W)$.
\end{corollary}
\begin{proof}
    The proof follows directly from \eqref{eq:toschedule}.
\end{proof}

 \begin{figure}[h]
	\centering
	\includegraphics[width=0.9\columnwidth ]{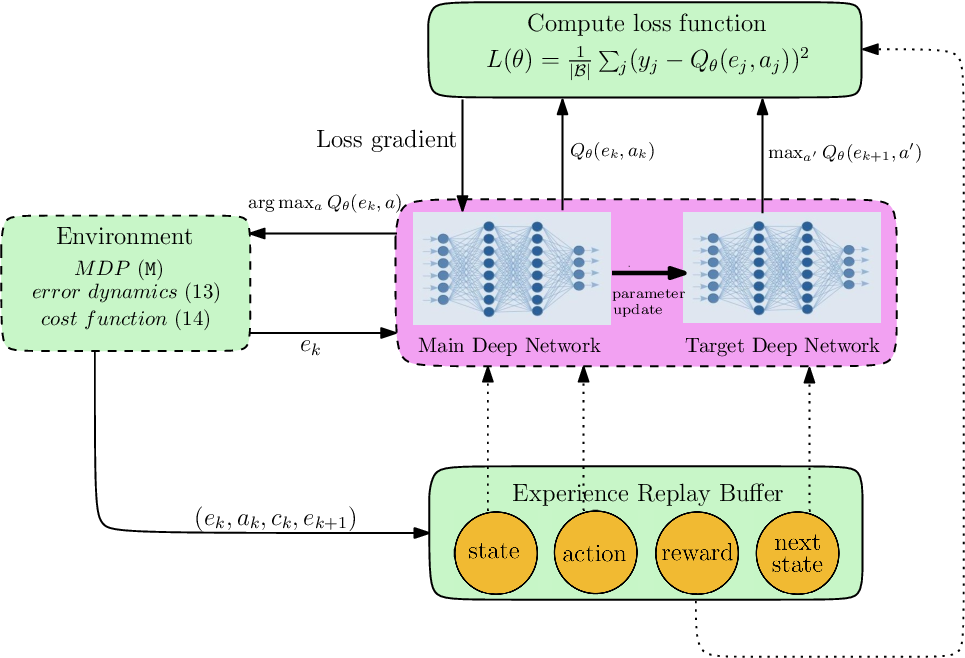}
 \vspace{-0.15cm}
	\caption{\small{Schematic representation of the \textbf{\texttt{InterQ}} Algorithm.
	}}
	\label{Fig:DQN_schematic}
 	\vspace{-0.2cm}
\end{figure}

We will later use this lemma  to validate the scheduling policy learned by the DQN framework. Next, we present our deep RL algorithm, \textbf{\texttt{InterQ}}, to compute the scheduling instances.

\subsection{\textbf{\texttt{InterQ}} for Scheduling Policy Computation}
Q-learning is an off-policy RL algorithm which estimates the optimal state-action value function using the Bellman equation. Deep Q-learning extends this approach by employing a DNN to approximate the Q-function. Unlike tabular Q-learning, which maintains a table of Q-values, deep Q-learning utilizes a function approximator $Q_\theta(e,a)$ with parameters $\theta$ (neural network weights) to estimate the Q-function. 

A major challenge in training a DNN with Q-learning is the correlation between consecutive experiences, which can lead to instability. Since neural networks typically require independent and identically distributed (i.i.d.) samples for effective training, sequential transitions can cause issues. To address this, {\tt Experience Replay} is used \cite{mnih2015human}, where a buffer $\mathcal{D}$ stores past experiences $(e,a,c,e')$ consisting of the error, action, incurred cost, and the next state. Instead of updating the Q-network using only the most recent experience, a mini-batch of samples is randomly drawn from the buffer for training. This technique reduces correlation between consecutive samples, improves sample efficiency by allowing experiences to be reused multiple times, and lowers the variance in weight updates, leading to a more stable learning.

Another challenge in Deep Q-learning is the moving target problem. Since the same Q-network is used both to estimate current Q-values and to compute target Q-values, the target shifts dynamically, leading to oscillations during learning. To mitigate this, a separate target Q-network is introduced, which is updated at a slower rate, typically after a fixed number of episodes. This stabilizes training by ensuring that target values change gradually.

A detailed algorithm as presented in Algorithm \ref{Fig:DQN_schematic} and its schematic is shown in Fig. \ref{DQN_algorithm}.
In summary, Algorithm \ref{DQN_algorithm} proceeds as follows: the process begins by initializing both the main Q-network and the target Q-network with random weights $\theta$ and $\theta'$, respectively, with $\theta = \theta'$ (lines 4-5). Next, an $\epsilon$--greedy policy is used to collect experiences (i.e., ($e,a,c,e'$)) based on system dynamics and the cost function, storing them in the replay buffer $\mathcal{D}$ (lines 9-12). A mini-batch is then sampled from $\mathcal{D}$, target values $y$ are computed, and the DNN is trained via gradient descent using a squared loss function with learning rate 
$\eta$ (lines 14-19). Finally, the target network is periodically updated at a fixed interval determined by the target update frequency $f_{\text{target}}$ (line 22).

\begin{algorithm}[h]
    \caption{\textbf{\texttt{InterQ}}: A DQN-based Intermittent Scheduling}
    \begin{algorithmic}[1]
    \label{DQN_algorithm}
        \STATE \textbf{Input:} System parameters: $A,B,Q,R, K_W,\lambda$, epoch length $T$, learning rate $\eta$, target update frequency $f_{\text{target}}$
        \STATE \textbf{Input: }exploration parameters $\epsilon = \epsilon_{\text{start}}, \epsilon_{\text{min}}, \epsilon_{\text{decay}}$
        \STATE \textbf{Input:} capacity $M$ of replay memory $\mathcal{D}$, {batch size} $b$ 
        \STATE Initialize action-value network $Q_\theta(e, a)$ with random weights $\theta$
        \STATE Initialize target network $Q_{\theta'}(e, a)$ with weights $\theta' = \theta$
\FOR{each episode }
        \STATE \textbf{Initialize} the estimation error $e_0$ randomly
    \FOR{ $k = 1: T$}
        \STATE \textbf{Select action} $a_k$ using $\epsilon$-greedy policy:
        \STATE \quad $a_k = \begin{cases} 
            \arg\min_{a} Q_\theta(e_k, a), & \!\!\text{with probability } 1 \!-\! \epsilon \\
            \text{random action}, & \!\! \text{with probability } \epsilon
        \end{cases}$
        \STATE Execute action $a_k$, observe cost $c_k$, generate a noise sample $W_k$, and obtain the next state $e_{k+1}$ using \eqref{eq:est_error}.
        \STATE Store experience $(e_k, a_k, c_k, e_{k+1})$ in $\mathcal{D}$
        \STATE \textbf{If} $M > b$:
        \STATE \quad Sample random mini-batch of transitions $(e_j, a_j, c_j, e_{j+1})$ from $\mathcal{D}$ of batch size $b$
        \STATE \quad Compute target Q-value for each sampled transition:
        \STATE \quad \quad $y_j = 
        \begin{cases} 
            c_j, & \text{if $j \!=\! T\!-\!1$}, \\
            c_j + \min\limits_{a'} Q_{\theta'}(e_{j+1}, a'), & \text{otherwise}.
        \end{cases}$
        \STATE \quad Compute least square loss function:
        
        \quad\quad $L(\theta) = \frac{1}{b} \sum_{j} (y_j - Q_\theta(e_j, a_j))^2$
        \STATE \quad Perform gradient descent on $\theta$: $\theta^+ = \theta - \eta \nabla_\theta L(\theta)$
        \STATE \quad Update exploration rate: $\epsilon \leftarrow \max(\epsilon_{\text{min}}, \epsilon \times \epsilon_{\text{decay}})$
    \ENDFOR
    \STATE \textbf{if} mod(episode,  $f_{\text{target}}) == 0$  
    \STATE \quad $\theta' \leftarrow \theta$
\ENDFOR
    \end{algorithmic}
\end{algorithm}

\section{Numerical Studies}\label{sec:numSims}
In this section, we corroborate the theoretical findings of the previous sections through simulations. For all simulations, we consider a two-dimensional unstable dynamics (i.e., $n_x=2$ in \eqref{eq:plant}) with the system matrices being $A =[1.5, 2; 0, 1.51],~B=[0;1]$, Gaussian noise with $K_W = [1, 0; 0, 1]$, the cost parameters being $Q = [1, 0; 0, 1]$, $ R=1$ and the discount factor $\gamma = 0.95$. Using the ARE \eqref{eq:ARE}, one can compute $P = [5.70, 7.34; 7.34, 14.36]$. Subsequently, we obtain  $\Gamma = 14.64$. We take the DNN to have 4 fully connected layers of dimensions $100 \times 100$ with GeLu (Gaussian error Linear units) activation functions. We set the memory size to 1000, the batch size to 16, the learning rate to 0.01, $\epsilon_{\text{start}} = 1.0$, $\epsilon_{\text{decay}} = 0.995$, $\epsilon_{\min} = 0.01$, and use \texttt{\textbf{Adam}} optimizer to train the DNN.

\textbf{1) Scheduling Landscape: } In our initial study, we examine the structure of the scheduling landscape, as illustrated in Fig.~\ref{Fig:sch_landscape}. 
To achieve this, we randomly generate multiple error vectors and determine the corresponding optimal scheduling actions using the trained DQN. 
Each chosen error state is plotted using a circular marker (see zoomed subfigure), where the colors of the circles correspond to the scheduling decisions returned by the trained network---magenta indicating a decision \textit{to schedule}, and cyan indicating \textit{not to schedule}.
Additionally, we overlay a (blue) ellipse on the plot, highlighting that the scheduling landscape closely resembles an elliptical shape (whose boundary is defined, for instance, by $e_k^\top Z e_k = 1$, for some matrix $Z$). Finally, we also superimpose the (red) ellipses generated by the sufficient conditions \eqref{eq:noschedule} and \eqref{eq:toschedule} which sandwich the blue ellipse, and further demonstrate the consistency of the solution proposed by \textbf{\texttt{InterQ}}.

\begin{figure}[t]
	\centering
	\includegraphics[width=0.7\columnwidth ]{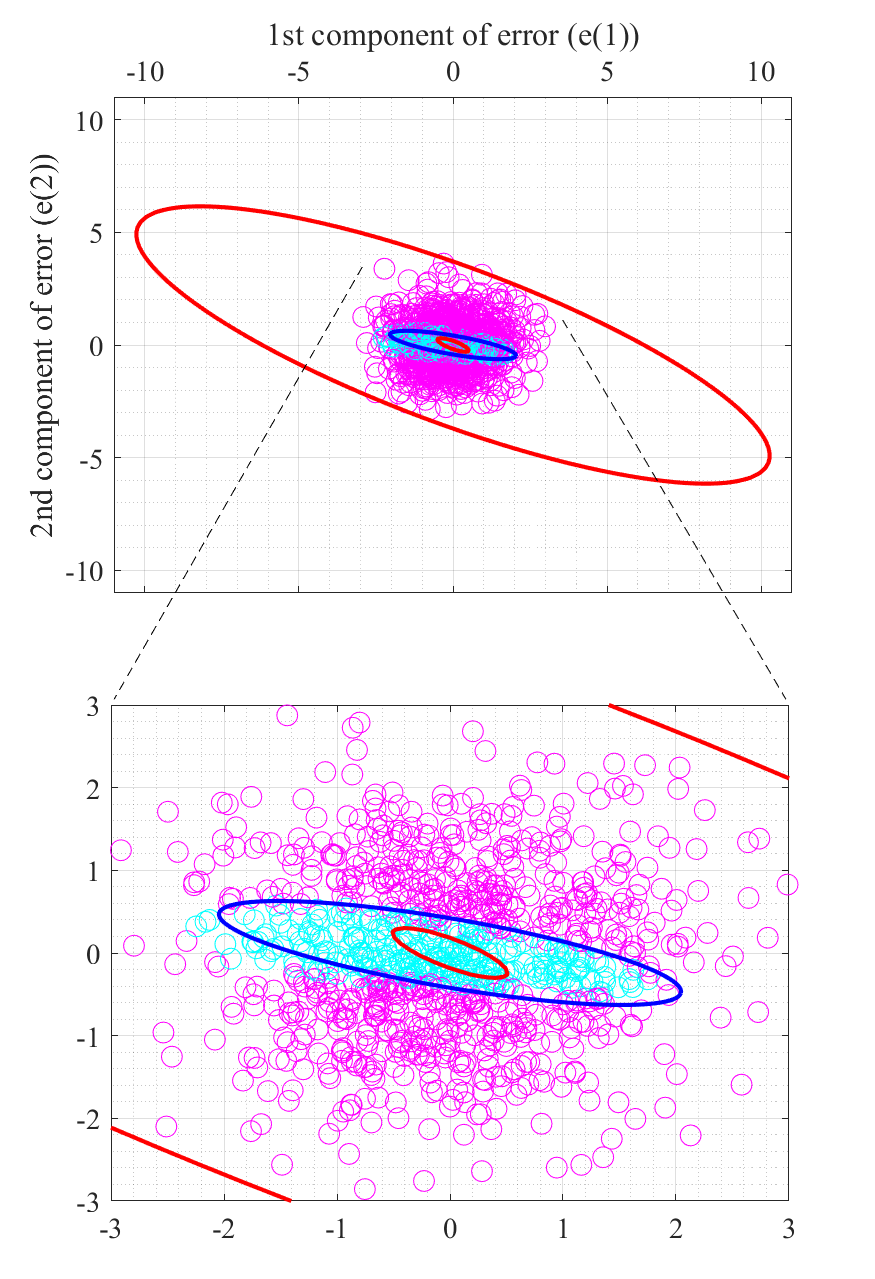}
 \vspace{-0.3cm}
	\caption{\small{Figure shows the communication-control trade-off curves (on the top) with its zoomed version (on the bottom); The outermost and innermost ellipses (in red) are plotted using \eqref{eq:toschedule} and \eqref{eq:noschedule}, respectively; the blue ellipse approximates the scheduling landscape generated by \textbf{\texttt{InterQ}}.
	}}
	\label{Fig:sch_landscape}
 	\vspace{-0.4cm}
\end{figure}

\textbf{2) Comparison with baselines:} In Fig.~\ref{Fig:comp_baseline}, we present a comparison of our algorithm with two baseline policies: \textbf{i) BS1:} periodic scheduling policy with period $\tau$, and \textbf{ii) BS2:} error-norm based triggering policy: $a_k = \mathbb I[\|e_k\|^2 \geq \tau]$.

%
\begin{figure}[h]
	\centering
	\includegraphics[width=\columnwidth ]{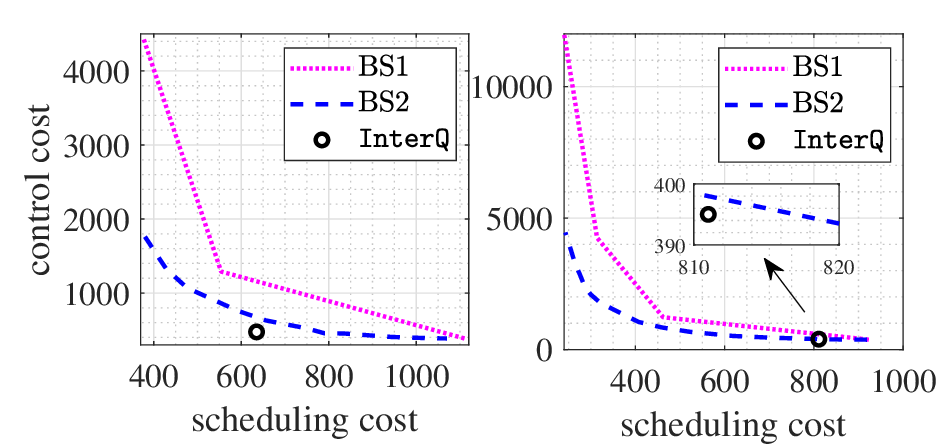}
 \vspace{-0.3cm}
	\caption{\small{Communication-control trade-off curves with $\lambda=50$ (on the left) and $\lambda=60$ (on the right).
	}}
	\label{Fig:comp_baseline}
 	\vspace{-0.3cm}
\end{figure}

In Fig.~\ref{Fig:comp_baseline}, we present two subfigures, each corresponding to a different value of $\lambda$. Each subfigure displays two Pareto-optimal curves that illustrate the trade-off between scheduling and control costs for the two baseline approaches. Our results indicate that \textbf{\texttt{InterQ}} consistently outperforms both BS1 and BS2, yielding a superior policy in terms of Pareto optimality.

\textbf{2) Effect of memory size $M$ on training stability:} In all experiments, a memory size of 1000 units is used. While a larger memory generally enhances the performance of the algorithm, we observed that a smaller value of $M$ can also be beneficial. This is because, for a continuous unbounded state space, a single poor experience $(e,a,c,e')$ can destabilize the loss function, making recovery difficult. A smaller memory size helps discard such detrimental experiences by replacing old ones with new, ensuring a more stable training. Therefore, selecting an appropriate memory size is crucial for achieving stable learning.

\textbf{3) Comparison with baselines for different noise distributions: }
Next, we also evaluate \textbf{\texttt{InterQ}} with the system noise distributed as a uniform random variable between -1 and 1, and compare our results with the baselines.
\begin{figure}[h]
\vspace{-2mm}
	\centering
	\includegraphics[width=0.56\columnwidth ]{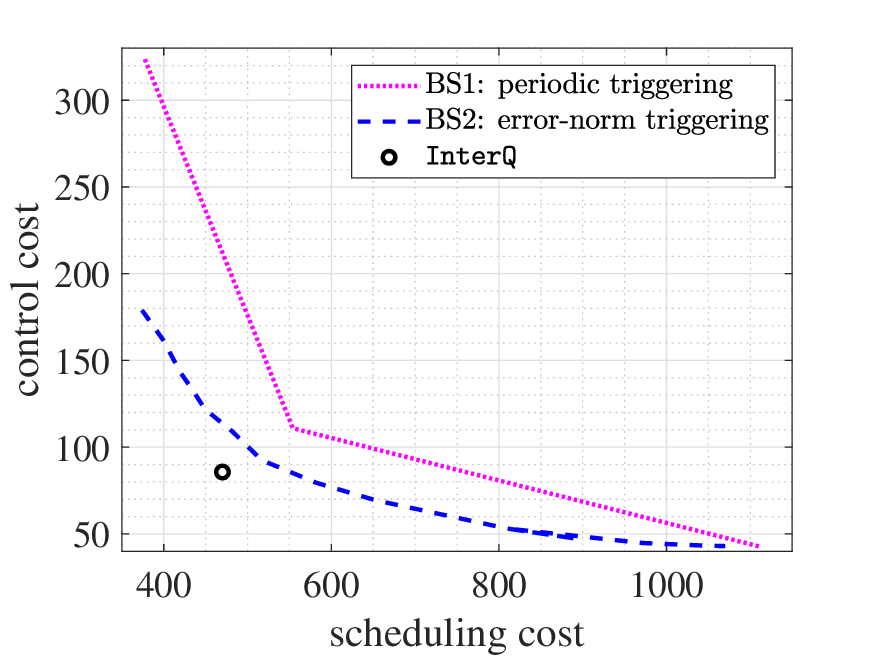}
 \vspace{-0.1cm}
	\caption{\small{Communication-control trade-off curves with $\lambda=60$ and system noise distributed as a uniform random variable.
	}}
	\label{Fig:comp_baseline}
\end{figure}
Again, we observe that our algorithm outperforms the results produced by the baselines.

\textbf{4) Effect of choice of loss function on training stability: }
Finally, we observed that using the built-in huber loss function \cite{rubio2023robust} in PyTorch leads to better training stability compared to the Mean Squared Error (MSE) loss. The Huber loss strikes a balance between the Mean Absolute Error (MAE) and MSE, making it more robust to outliers--an important consideration given the presence of Gaussian noise in our setting. Unlike MSE, which heavily penalizes large errors, huber loss reduces the impact of extreme values, leading to more stable training and improved scheduling policy performance.

This completes our numerical analysis of the proposed deep RL approach.
\vspace{-2mm}
\section{Conclusions}\label{sec:conc_disc}

In this \textit{letter}, we have addressed the communication-control co-design problem using deep reinforcement learning. Specifically, we first presented the controller-estimator design by invoking the separation principle. Next, we analyzed the qualitative properties of the scheduling policy. Finally, we introduced \textbf{\texttt{InterQ}}, a deep Q-learning-based framework for determining optimal scheduling instances. Through extensive numerical evaluations, we validated our theoretical findings and demonstrated the superiority of our approach over multiple baseline methods.

An interesting future direction would be to incorporate the \textit{no-communication events} into the information set of the controller. This makes the design of the controller and the scheduler coupled, and hence, the estimator is no longer linear. Thus, it would be interesting to extend our RL framework toward joint design for both linear and non-linear systems.

\bibliography{references,refs}
\bibliographystyle{IEEEtran}

\end{document}